\documentclass[12pt]{amsart}
\usepackage{mathrsfs}
\usepackage{mathabx}
\usepackage{enumerate}
\usepackage{skak}
\usepackage{bbm}
\usepackage{skull}
\usepackage{mathrsfs, mathtools}
\usepackage{stmaryrd}
\usepackage{amsmath,amssymb,amsfonts, mathabx}
\usepackage[all,cmtip]{xy}
\usepackage{a4wide}
\usepackage[mathscr]{eucal}
\usepackage[toc,page]{appendix}
\usepackage{verbatim}
\usepackage{dsfont}

\usepackage{extarrow}
\usepackage{enumerate}
\usepackage{fontenc}
\usepackage[T1]{fontenc}
\usepackage{graphicx}
\usepackage[colorlinks=true, linkcolor=blue]{hyperref}
\usepackage{latexsym}
\usepackage{mathrsfs}
\usepackage{mathtext}
\usepackage{tikz}
\usepackage{textcomp}
\usepackage{upgreek}
\usepackage{xy}

\usetikzlibrary{arrows}
\usetikzlibrary{matrix}
\usetikzlibrary{shapes}
\usetikzlibrary{snakes}
\usetikzlibrary{matrix}

\DeclareMathOperator{\Add}{\textup{Add}}
\DeclareMathOperator{\Alg}{\textup{Alg}}
\DeclareMathOperator{\Ann}{\textup{Ann}}
\DeclareMathOperator{\Arr}{\textup{Arr}}
\DeclareMathOperator{\Art}{\textup{Art}}
\DeclareMathOperator{\Ass}{\textup{Ass}}
\DeclareMathOperator{\Autsh}{\underline{\textup{Aut}}}
\DeclareMathOperator{\Bi}{\textup{B}}
\DeclareMathOperator{\CAdd}{\textup{CAdd}}
\DeclareMathOperator{\CAlg}{\textup{CAlg}}
\DeclareMathOperator{\CMon}{\textup{CMon}}
\DeclareMathOperator{\CPMon}{\textup{CPMon}}
\DeclareMathOperator{\CRings}{\textup{CRings}}
\DeclareMathOperator{\CSMon}{\textup{CSMon}}
\DeclareMathOperator{\CaCl}{\textup{CaCl}}
\DeclareMathOperator{\Cart}{\textup{Cart}}
\DeclareMathOperator{\Cl}{\textup{Cl}}
\DeclareMathOperator{\Coh}{\textup{Coh}}
\DeclareMathOperator{\Coker}{\textup{Coker}}
\DeclareMathOperator{\Der}{\textup{Der}}
\DeclareMathOperator{\End}{\textup{End}}
\DeclareMathOperator{\Endsh}{\underline{\textup{End}}}
\DeclareMathOperator{\Ext}{\textup{Ext}}
\DeclareMathOperator{\Extsh}{\underline{\textup{Ext}}}
\DeclareMathOperator{\FAdd}{\textup{FAdd}}
\DeclareMathOperator{\FCoh}{\textup{FCoh}}
\DeclareMathOperator{\FGrad}{\textup{FGrad}}
\DeclareMathOperator{\FLoc}{\textup{FLoc}}
\DeclareMathOperator{\FMod}{\textup{FMod}}
\DeclareMathOperator{\FPMon}{\textup{FPMon}}
\DeclareMathOperator{\FRep}{\textup{FRep}}
\DeclareMathOperator{\FSMon}{\textup{FSMon}}
\DeclareMathOperator{\FVect}{\textup{FVect}}
\DeclareMathOperator{\Fibr}{\textup{Fibr}}
\DeclareMathOperator{\Fix}{\textup{Fix}}
\DeclareMathOperator{\Fl}{\textup{Fl}}
\DeclareMathOperator{\Fr}{\textup{Fr}}
\DeclareMathOperator{\Funct}{\textup{Funct}}
\DeclareMathOperator{\GAlg}{\textup{GAlg}}
\DeclareMathOperator{\GExt}{\textup{GExt}}
\DeclareMathOperator{\GHom}{\textup{GHom}}
\DeclareMathOperator{\GL}{\textup{GL}}
\DeclareMathOperator{\GMod}{\textup{GMod}}
\DeclareMathOperator{\GRis}{\textup{GRis}}
\DeclareMathOperator{\GRiv}{\textup{GRiv}}
\DeclareMathOperator{\Gl}{\textup{Gl}}
\DeclareMathOperator{\Grad}{\textup{Grad}}
\DeclareMathOperator{\Hilb}{\textup{Hilb}}
\DeclareMathOperator{\Hl}{\textup{H}}
\DeclareMathOperator{\Homsh}{\underline{\textup{Hom}}}
\DeclareMathOperator{\ISym}{\textup{Sym}^*}
\DeclareMathOperator{\Imm}{\textup{Im}}
\DeclareMathOperator{\Irr}{\textup{Irr}}
\DeclareMathOperator{\Iso}{\textup{Iso}}
\DeclareMathOperator{\Isosh}{\underline{\textup{Iso}}}
\DeclareMathOperator{\LAdd}{\textup{LAdd}}
\DeclareMathOperator{\LAlg}{\textup{LAlg}}
\DeclareMathOperator{\LMon}{\textup{LMon}}
\DeclareMathOperator{\LPMon}{\textup{LPMon}}
\DeclareMathOperator{\LRings}{\textup{LRings}}
\DeclareMathOperator{\LSMon}{\textup{LSMon}}
\DeclareMathOperator{\Left}{\textup{L}}
\DeclareMathOperator{\Lex}{\textup{Lex}}
\DeclareMathOperator{\ML}{\textup{ML}}
\DeclareMathOperator{\MLex}{\textup{MLex}}
\DeclareMathOperator{\Mon}{\textup{Mon}}
\DeclareMathOperator{\Ob}{\textup{Ob}}
\DeclareMathOperator{\Obj}{\textup{Obj}}
\DeclareMathOperator{\PDiv}{\textup{PDiv}}
\DeclareMathOperator{\PGL}{\textup{PGL}}
\DeclareMathOperator{\PML}{\textup{PML}}
\DeclareMathOperator{\PMLex}{\textup{PMLex}}
\DeclareMathOperator{\PMon}{\textup{PMon}}
\DeclareMathOperator{\Picsh}{\underline{\textup{Pic}}}
\DeclareMathOperator{\Pro}{\textup{Pro}}
\DeclareMathOperator{\Proj}{\textup{Proj}}
\DeclareMathOperator{\QAdd}{\textup{QAdd}}
\DeclareMathOperator{\QAlg}{\textup{QAlg}}
\DeclareMathOperator{\QCoh}{\textup{QCoh}}
\DeclareMathOperator{\QMon}{\textup{QMon}}
\DeclareMathOperator{\QPMon}{\textup{QPMon}}
\DeclareMathOperator{\QRings}{\textup{QRings}}
\DeclareMathOperator{\QSMon}{\textup{QSMon}}
\DeclareMathOperator{\Rings}{\textup{Rings}}
\DeclareMathOperator{\Riv}{\textup{Riv}}
\DeclareMathOperator{\SFibr}{\textup{SFibr}}
\DeclareMathOperator{\SMLex}{\textup{SMLex}}
\DeclareMathOperator{\SMex}{\textup{SMex}}
\DeclareMathOperator{\SMon}{\textup{SMon}}
\DeclareMathOperator{\SchI}{\textup{SchI}}
\DeclareMathOperator{\Sh}{\textup{Sh}}
\DeclareMathOperator{\Soc}{\textup{Soc}}
\DeclareMathOperator{\Specsh}{\underline{\textup{Spec}}}
\DeclareMathOperator{\Stab}{\textup{Stab}}
\DeclareMathOperator{\Supp}{\textup{Supp}}
\DeclareMathOperator{\Sym}{\textup{Sym}}
\DeclareMathOperator{\TMod}{\textup{TMod}}
\DeclareMathOperator{\Top}{\textup{Top}}
\DeclareMathOperator{\Tor}{\textup{Tor}}
\DeclareMathOperator{\alt}{\textup{ht}}
\DeclareMathOperator{\car}{\textup{char}}
\DeclareMathOperator{\degtr}{\textup{degtr}}
\DeclareMathOperator{\depth}{\textup{depth}}
\DeclareMathOperator{\divis}{\textup{div}}
\DeclareMathOperator{\et}{\textup{et}}
\DeclareMathOperator{\ffpSch}{\textup{ffpSch}}
\DeclareMathOperator{\h}{\textup{h}}
\DeclareMathOperator{\ilim}{\displaystyle{\lim_{\longrightarrow}}}
\DeclareMathOperator{\indim}{\textup{inj dim}}
\DeclareMathOperator{\lf}{\textup{LF}}
\DeclareMathOperator{\op}{\textup{op}}
\DeclareMathOperator{\ord}{\textup{ord}}
\DeclareMathOperator{\pd}{\textup{pd}}
\DeclareMathOperator{\plim}{\displaystyle{\lim_{\longleftarrow}}}
\DeclareMathOperator{\pr}{\textup{pr}}
\DeclareMathOperator{\pt}{\textup{pt}}
\DeclareMathOperator{\rk}{\textup{rk}}
\DeclareMathOperator{\tr}{\textup{tr}}
\DeclareMathOperator{\type}{\textup{r}}
\DeclareMathOperator*{\colim}{\textup{colim}}
\usepackage[colorinlistoftodos]{todonotes}

\theoremstyle{plain}
\newtheorem{thm}{Theorem}[section]
\newtheorem{lem}[thm]{Lemma}

\newtheorem{prop}[thm]{Proposition}

\theoremstyle{plain}
\newtheorem{thmI}{Theorem}
\newtheorem{thmII}{Theorem}

\theoremstyle{definition}

\newtheorem{defn}[thm]{Definition} 

\newtheorem{rmk}[thm]{Remark}

\numberwithin{thm}{section}
\newcounter{x}

\newcommand{\red}{{\rm red}}

\newcommand{\EF}{{\rm EFin}}
\newcommand{\Et}{{\rm Et}}
\newcommand{\codim}{{\rm codim}}

\newcommand{\Pic}{{\rm Pic}}

\newcommand{\Div}{{\rm Div}}

\newcommand{\Hom}{{\rm Hom}}

\newcommand{\Loc}{{\rm Loc}}

\newcommand{\Spec}{{\rm Spec \,}}

\newcommand{\Aff}{{\rm Aff}}

\newcommand{\Gal}{{\rm Gal}}

\newcommand{\Ker}{{\rm Ker}}
\newcommand{\Aut}{{\rm Aut}}

\newcommand{\sB}{{\mathcal B}}

\newcommand{\sO}{{\mathcal O}}

\newcommand{\sS}{{\mathcal S}}

\newcommand{\sX}{{\mathcal X}}
\newcommand{\sY}{{\mathcal Y}}
\newcommand{\sZ}{{\mathcal Z}}
\newcommand{\A}{{\mathbb A}}

\newcommand{\C}{{\mathbb C}}

\newcommand{\E}{{\mathbb E}}
\newcommand{\F}{{\mathbb F}}

\newcommand{\M}{{\mathbb M}}
\newcommand{\N}{{\mathbb N}}

\newcommand{\Q}{{\mathbb Q}}
\newcommand{\R}{{\mathbb R}}

\newcommand{\Z}{{\mathbb Z}}

\newcommand{\NN}{\textup{N}}

\newcommand{\Mod}{\text{\sf Mod}}

\newcommand{\Vect}{\text{\sf Vect}}
\newcommand{\Rep}{\text{\sf Rep}}
\newcommand{\id}{{\rm id\hspace{.1ex}}}

\newcommand{\ind}{{\text{\sf ind}\hspace{.1ex}}}
\newcommand{\Cov}{{\mathcal Cov}}

\begin{document}
\title{Essentially Finite Vector Bundles on  Normal Pseudo-proper Algebraic Stacks}

\author{Fabio Tonini,  Lei Zhang }
\address{ Fabio Tonini\\
    Freie Universit\"at Berlin\\
    FB Mathematik und Informatik\\
    Arnimallee 3\\ Zimmer 112A\\
    14195 Berlin\\ Deutschland }
\email{tonini@math.hu-berlin.de}
 \address{ Lei Zhang\\
    Freie Universit\"at Berlin\\
    FB Mathematik und Informatik\\
    Arnimallee 3\\ Zimmer 112A\\
    14195 Berlin\\ Deutschland }
\email{l.zhang@fu-berlin.de}

\thanks{This work was supported by the European Research Council (ERC) Advanced Grant 0419744101 and the Einstein Foundation}
\date{\today}

\global\long\def\A{\mathbb{A}}

\global\long\def\Ab{(\textup{Ab})}

\global\long\def\C{\mathbb{C}}

\global\long\def\Cat{(\textup{cat})}

\global\long\def\Di#1{\textup{D}(#1)}

\global\long\def\E{\mathcal{E}}

\global\long\def\F{\mathbb{F}}

\global\long\def\GCov{G\textup{-Cov}}

\global\long\def\Gcat{(\textup{Galois cat})}

\global\long\def\Gfsets#1{#1\textup{-fsets}}

\global\long\def\Gm{\mathbb{G}_{m}}

\global\long\def\GrCov#1{\textup{D}(#1)\textup{-Cov}}

\global\long\def\Grp{(\textup{Grps})}

\global\long\def\Gsets#1{(#1\textup{-sets})}

\global\long\def\HCov{H\textup{-Cov}}

\global\long\def\MCov{\textup{D}(M)\textup{-Cov}}

\global\long\def\MHilb{M\textup{-Hilb}}

\global\long\def\N{\mathbb{N}}

\global\long\def\PGor{\textup{PGor}}

\global\long\def\PGrp{(\textup{Profinite Grp})}

\global\long\def\PP{\mathbb{P}}

\global\long\def\Pj{\mathbb{P}}

\global\long\def\Q{\mathbb{Q}}

\global\long\def\RCov#1{#1\textup{-Cov}}

\global\long\def\RR{\mathbb{R}}

\global\long\def\Sch{\textup{Sch}}

\global\long\def\WW{\textup{W}}

\global\long\def\Z{\mathbb{Z}}

\global\long\def\acts{\curvearrowright}

\global\long\def\alA{\mathscr{A}}

\global\long\def\alB{\mathscr{B}}

\global\long\def\arr{\longrightarrow}

\global\long\def\arrdi#1{\xlongrightarrow{#1}}

\global\long\def\catC{\mathscr{C}}

\global\long\def\catD{\mathscr{D}}

\global\long\def\catF{\mathscr{F}}

\global\long\def\catG{\mathscr{G}}

\global\long\def\comma{,\ }

\global\long\def\covU{\mathcal{U}}

\global\long\def\covV{\mathcal{V}}

\global\long\def\covW{\mathcal{W}}

\global\long\def\duale#1{{#1}^{\vee}}

\global\long\def\fasc#1{\widetilde{#1}}

\global\long\def\fsets{(\textup{f-sets})}

\global\long\def\iL{r\mathscr{L}}

\global\long\def\id{\textup{id}}

\global\long\def\la{\langle}

\global\long\def\odi#1{\mathcal{O}_{#1}}

\global\long\def\ra{\rangle}

\global\long\def\set{(\textup{Sets})}

\global\long\def\sets{(\textup{Sets})}

\global\long\def\shA{\mathcal{A}}

\global\long\def\shB{\mathcal{B}}

\global\long\def\shC{\mathcal{C}}

\global\long\def\shD{\mathcal{D}}

\global\long\def\shE{\mathcal{E}}

\global\long\def\shF{\mathcal{F}}

\global\long\def\shG{\mathcal{G}}

\global\long\def\shH{\mathcal{H}}

\global\long\def\shI{\mathcal{I}}

\global\long\def\shJ{\mathcal{J}}

\global\long\def\shK{\mathcal{K}}

\global\long\def\shL{\mathcal{L}}

\global\long\def\shM{\mathcal{M}}

\global\long\def\shN{\mathcal{N}}

\global\long\def\shO{\mathcal{O}}

\global\long\def\shP{\mathcal{P}}

\global\long\def\shQ{\mathcal{Q}}

\global\long\def\shR{\mathcal{R}}

\global\long\def\shS{\mathcal{S}}

\global\long\def\shT{\mathcal{T}}

\global\long\def\shU{\mathcal{U}}

\global\long\def\shV{\mathcal{V}}

\global\long\def\shW{\mathcal{W}}

\global\long\def\shX{\mathcal{X}}

\global\long\def\shY{\mathcal{Y}}

\global\long\def\shZ{\mathcal{Z}}

\global\long\def\st{\ | \ }

\global\long\def\stA{\mathcal{A}}

\global\long\def\stB{\mathcal{B}}

\global\long\def\stC{\mathcal{C}}

\global\long\def\stD{\mathcal{D}}

\global\long\def\stE{\mathcal{E}}

\global\long\def\stF{\mathcal{F}}

\global\long\def\stG{\mathcal{G}}

\global\long\def\stH{\mathcal{H}}

\global\long\def\stI{\mathcal{I}}

\global\long\def\stJ{\mathcal{J}}

\global\long\def\stK{\mathcal{K}}

\global\long\def\stL{\mathcal{L}}

\global\long\def\stM{\mathcal{M}}

\global\long\def\stN{\mathcal{N}}

\global\long\def\stO{\mathcal{O}}

\global\long\def\stP{\mathcal{P}}

\global\long\def\stQ{\mathcal{Q}}

\global\long\def\stR{\mathcal{R}}

\global\long\def\stS{\mathcal{S}}

\global\long\def\stT{\mathcal{T}}

\global\long\def\stU{\mathcal{U}}

\global\long\def\stV{\mathcal{V}}

\global\long\def\stW{\mathcal{W}}

\global\long\def\stX{\mathcal{X}}

\global\long\def\stY{\mathcal{Y}}

\global\long\def\stZ{\mathcal{Z}}

\global\long\def\then{\ \Longrightarrow\ }

\global\long\def\L{\textup{L}}

\global\long\def\l{\textup{l}}

\makeatletter 
\providecommand\@dotsep{5} 
\makeatother 

\setcounter{section}{0}
\maketitle
\begin{abstract} Let $X$ be a normal, connected and projective variety over an algebraically closed field $k$. In \cite{BDS} and \cite{AM} it is proved that a vector bundle $V$ on $X$ is essentially finite if and only if it is trivialized by a proper surjective morphism $f:Y\arr X$. In this paper we introduce a different approach to this problem which allows to extend the results to normal, connected and strongly pseudo-proper algebraic stack of finite type over an arbitrary field $k$.
\end{abstract}
\section*{Introduction}
Let $k$ be a base field and $X$ be a proper, connected and reduced scheme over $k$ with a rational point $x\in X(k)$. In \cite{Nori} M. Nori introduced the \emph{Nori fundamental group scheme} $\pi^\NN(X,x)$, which classifies torsors over $X$ under finite $k$-group schemes and with a trivialization over $x$, and proved that the category $\Rep (\pi^\NN(X,x))$ of its finite $k$-representations can be identified with the subcategory of $\Vect(\stX)$ of vector bundles which are \emph{essentially finite} (see \ref{essentially finite}).

In \cite{BDS} and \cite{BDS2}, I. Biswas and J. P. Dos Santos gave a more geometric characterization of essentially finite vector bunldles. If $X$ is a smooth, connected and projective variety over an algebraically closed field they showed that a vector bundle $V$ on $X$ is essentially finite if and only if it is trivialized by a proper surjective morphism $f:Y\arr X$, that is $f^*V$ is a free vector bundle. This result has then be generalized to normal varieties in \cite{AM}. In this paper we present a new proof of this result which apply to more general $X$ and does not require the use of semistable sheaves. Let us introduce some notions before stating our results.

A category $\sX$ fibered in groupoid over a field $k$ is \textit{pseudo-proper} (resp. \textit{strongly pseudo-proper}) if for all vector bundles (resp.  finitely presented sheaves) $E$ on $\sX$ the $k$-vector space $\Hl^0(\sX,E)$ is finite dimensional. It is also required to satisfy a finiteness condition which is automatic for algebraic stacks of finite type over $k$ (see \ref{infl and pseudo-proper}).
Proper algebraic stacks, finite stacks and affine gerbes are examples of strongly pseudo-proper fibred categories (see \cite[Example 7.2, pp. 20]{BV}). This is the generalization of \cite{BDS} and \cite{AM} we present.

\begin{thmI}\label{Dos Santos and Biswas} Let $\sX$ be a connected, normal and strongly pseudo-proper algebraic stack of finite type over $k$. Then a vector bundle $V$ on $\sX$ is essentially finite if and only if it is trivialized by a surjective morphism $f: \sY\arr \sX$ between algebraic stacks such that $f_*\sO_{\sY}$ is a coherent sheaf. 
\end{thmI}

We stress that the proof of this result is stacky in nature and not only because we already start with stacks. The key example that enlightens the strategy is  when $X=\stX$ is a normal variety and $f\colon Y=\stY\arr \sX$ is the normalization of $X$ in a Galois extension $L/k(X)$ with group $G$. Then we have a splitting $Y\arr \stZ=[Y/G] \arrdi \pi X$. Since $Y$ is a $G$-torsor over  $\sZ$ and $\pi^*V$ is trivialized by this torsor, $\pi^*V$ is essentially finite. The conclusion that $V$ is essentially finite  then follows formally from the fact that $\pi_*\odi\stZ\simeq \odi X$. Thus the key step here is to pass to a new space $\sZ$ which may not be a scheme. 
When  $f$ is general one can reduce the problem to the above situation. This is not possible for stacks and in this case we introduce a different kind of ``Galois cover'' where the group $G$ is replaced by the symmetric group.

In \cite[Corollary I]{TZ2} we prove a variant of Theorem \ref{Dos Santos and Biswas} without any regularity requirement on $\stX$ but assuming that $f$ is proper and \textit{flat}. The proofs of those results are independent.

For completeness we also deal with the analogous result of \cite{BDS2} in our context:

\begin{thmII}\label{Dos Santos and Biswas0} Let $\sX$ be a connected, normal and strongly pseudo-proper algebraic stack of finite type over $k$ and $f : \sY \arr \sX$ be a surjective morphism of algebraic stacks such that $f_*\sO_{\sY}$ is a coherent sheaf. Consider the full subcategory of $\Vect(\stX)$
$$\Vect(\sX)_f :=\{\ V \in \Vect(\sX)\ |\  f^*V\text{ is trivial }\}$$
and set $L=\Hl^0(\stX,\odi\stX)$, which is a finite field extension of $k$. We have
\begin{enumerate}
\item if $\Spec(f_*\sO_{\sY})$ is reduced (e.g. if $\stY$ is reduced), then $\Vect(\stX)_f$ is an $L$-Tannakian category;
\item  if the generic fiber of $\Spec(f_*\sO_{\sY})\arr \sX$ is \'etale then the affine gerbe over $L$ corresponding to the full sub tannakian category generated by $\Vect(\sX)_f\subseteq \EF(\Vect(\stX))$ is finite and \'etale.
\end{enumerate}
\end{thmII}

Notice that in the above hypothesis the category $\EF(\Vect(\stX))$ is indeed an $L$-Tannakian category (see \ref{inflexible}). In situation $(1)$ the gerbe associated with $\Vect(\stX)_f$ is in general not finite, as already shown in \cite{BDS2}.

\section*{Acknowledgement}
 We would like to thank B. Bhatt, N. Borne, H. Esnault, M. Olsson, M. Romagny and A. Vistoli for helpful conversations and suggestions received.

\section{Preliminaries}
In this section we fix a base field $k$. We will collect here some preliminary results and definitions needed for the next section. All fibered categories considered will be fibered in groupoids.

We will freely talk about affine gerbes over a field (often improperly called just gerbes) and Tannakian categories and use their properties. Please refer to \cite[Appendix B]{TZ} for details.

\begin{defn}\label{essentially finite}\cite[Definition 7.7]{BV}
Let $\shC$ be an additive and monoidal category. An object $E\in\shC$ is called finite if there exist $f\neq g \in \N[X]$ polynomials with natural coefficients and an isomorphism $f(E)\simeq g(E)$, it is called \emph{essentially finite} if it is a kernel of a map of finite objects of $\shC$. We denote by $\EF(\shC)$ the full subcategory of $\shC$ consisting of essentially finite objects.
\end{defn}

\begin{defn}\label{infl and pseudo-proper}\cite[Definition 5.3 and Definition 7.1]{BV}
 Let $\stX$ be a fibred category over $k$. It is called \emph{inflexible} over $k$ if any map from $\stX$ to a finite stack factors through a finite gerbe. It is called  \emph{pseudo-proper} (resp. \emph{strongly pseudo-proper}) if
  \begin{enumerate}
\item there exists a quasi-compact scheme $U$ and a representable morphism $U\arr X$ which is faithfully flat, quasi-compact and quasi-separated;
\item for any vector bundle (resp. finitely presented quasi-coherent sheaf) $E$ on $\sX$ the $k$-vector space $\Hl^0(\sX,E)$ is finite dimensional.
\end{enumerate}  
\end{defn}

\begin{rmk}\label{inflexible}
 By \cite[Theorem 7.9, pp. 22]{BV}, if $\stX$ is a pseudo-proper and inflexible fibred category then $\EF(\Vect(\stX))$ is a $k$-Tannakian category, which corresponds to the so called Nori fundamental gerbe $\Pi^\NN_{\stX/k}$.
 
 Recall that a pseudo-proper fibred category $\stX$ which is inflexible satisfies $\Hl^0(\odi\stX)=k$ and the converse is true if $\stX$ is reduced, quasi-compact and quasi-separated (see \cite[Theorem 4.4]{TZ}). In particular in Theorem \ref{Dos Santos and Biswas} and \ref{Dos Santos and Biswas0} the algebraic stack $\stX$ is automatically inflexible over the field $\Hl^0(\odi\stX)$.
\end{rmk}

Recall that if $\stX$ is an algebraic stack with a map $\lambda\colon \stX\arr \Gamma$ to a finite gerbe and $W\in \Vect(\Gamma)$ then $\lambda^*W$ is essentially finite. Indeed all vector bundles on $\Gamma$ are essentially finite by \cite[Proposition 7.8]{BV} and $\lambda^*$ is exact and monoidal because $\lambda$ is flat. 

We start by looking at a special case of Theorem \ref{Dos Santos and Biswas}, that is the case of a torsor.

\begin{lem} \label{finite etale trivialization}
Let $\sX$ be a pseudo-proper and inflexible fibered category over $k$ and let $f\colon \stY\arr\stX$ be a torsor under some finite $k$-group scheme $G$. If $V$ is a vector bundle on $\sX$ which is trivialized by $f$ then  $V$ is an essentially finite vector bundle on $\sX$. Moreover the full subcategory of $\Vect(\stX)$ of vector bundles trivialized by $f$ is a $k$-Tannakian category whose associated gerbe $\Gamma$ is the Nori reduction of $\stX\arr \sB_k G$, that is,  it is the image of the unique map $\Pi_{\sX/k}^N\arr \sB_kG$ which corresponds to $\sX\arr \sB_k G$.
\end{lem}
\begin{proof}
Consider the following $2$-cartesian diagram 
$$\xymatrix{\sY\ar[rr]\ar[d]^-f&& \Spec(k)\ar[d]\\ \sX\ar[rr]& & \sB_kG}$$
By \cite[Lemma 5.12, Lemma 7.11]{BV} there is a finite gerbe $\Gamma$ and a factorization of $\stX\arr \sB_k G$ as: $\lambda\colon \stX\arr \Gamma$ such that $\lambda_*\odi\stX\simeq\odi\Gamma$; $\Gamma\arr\sB_k G$ faithful. Notice that, since $\Gamma$ is finite, the map $\Gamma\arr \sB_k G$ is affine. See for instance \cite[Remark B7]{TZ}.
In particular we also have a $2$-Cartesian diagram
$$\xymatrix{\sY\ar[rr]^-{\delta}\ar[d]^-f&& \Spec(A)\ar[d]\\ \sX\ar[rr]^{\lambda}& & \Gamma}$$
By cohomology and base change along the flat map $\Spec A\arr \Gamma$ we have $\delta_*\sO_{\sY}\cong\sO_{\Spec(A)}$. Pulling back the adjunction $\lambda^*\lambda_*V\arr V$ along $f$ we get the  map $\delta^*\delta_*f^*V\cong f^*\lambda^*\lambda_*V\arr f^*V$, which is an isomorphism as $f^*V\cong \sO_{\sY}^{\oplus m}$ and $\delta_*\sO_{\sY}\cong\sO_{\Spec(A)}$. Since $f$ is faithfully flat it follows that $\lambda^*\lambda_*V\arr V$ is an isomorphism. Thus $V$ is the pull back of a vector bundle on a finite gerbe and hence it is essentially finite.

The map $\Vect(\Gamma)\arr\Vect(\stX)$ is fully faithful and embeds $\Vect(\Gamma)$ as a sub Tannakian category of $\EF(\Vect(\stX))$ by   \cite[Theorem 7.13, pp. 24]{BV}. To conclude the proof we just have to show that if $W\in \Vect(\Gamma)$ then $f^*\lambda^*W$ is free. Thus it is enough to show that $A$ is a finite $k$-algebra and we can assume $k$ algebraically closed. In this case $\Gamma=\sB_k H$, where $H$ is a closed subgroup of $G$, and a direct computation shows that $\Spec A\simeq G/H$ which is a finite $k$-scheme.
\end{proof}

The following result is also a variant of Theorem \ref{Dos Santos and Biswas}, which is useful because it will allow us to replace an arbitrary finite map by a generically \'etale one. The same result is present in \cite[Lemma 2.3]{TZ2}. We include it here for completeness.
\begin{lem}\label{Frobenius pullback of essentially finite}
Let $\stX$ be an inflexible and pseudo-proper fibred category over $k$, $V\in \Vect(\stX)$ and denote by $F\colon \stX\arr\stX$ the absolute Frobenius. If there exists $m\in \N$ such that $F^{m*}V\in \Vect(\stX)$ is essentially finite then $V$ is essentially finite too.
\end{lem}
\begin{proof}
 We can consider the case $m=1$ only. Set $n=\rk V$. The vector bundle $V$ is given by an  $\F_p$-map $v\colon \stX\arr \sB_{\F_p}(\GL_{n,\F_p})$: the stack $\sB_{\F_p}(\GL_{n,\F_p})$ has a universal vector bundle $\E$ of rank $n$ such that $v^*\E\simeq V$. By  \ref{inflexible} $V$ is essentially finite if and only if $v$ factors as $\stX\arrdi\phi \Gamma \arr \sB_{\F_p}(\GL_{n,\F_p})$ where $\Gamma$ is a finite $k$-gerbe and $\phi$ is $k$-linear. The vector bundle $F^*V$ corresponds to the composition $\stX\arrdi v \sB_{\F_p}(\GL_{n,\F_p})\arrdi{\overline F} \sB_{\F_p}(\GL_{n,\F_p})$, where $\overline F$ is the absolute Frobenius of $\sB_{\F_p}(\GL_{n,\F_p})$. Thus we have a diagram
   \[
  \begin{tikzpicture}[xscale=2.0,yscale=-1.2]
    \node (A0_0) at (0, 0) {$\stX$};
    \node (A0_1) at (1, 0) {$\Delta$};
    \node (A0_2) at (2, 0) {$\sB_{\F_p}(\GL_{n,\F_p})$};
    \node (A1_1) at (1, 1) {$\Gamma$};
    \node (A1_2) at (2, 1) {$\sB_{\F_p}(\GL_{n,\F_p})$};
    \path (A0_0) edge [->]node [auto] {$\scriptstyle{}$} (A0_1);
    \path (A0_0) edge [swap,->]node [auto] {$\scriptstyle{\phi}$} (A1_1);
    \path (A0_1) edge [->]node [auto] {$\scriptstyle{}$} (A1_1);
    \path (A0_2) edge [->]node [auto] {$\scriptstyle{\overline F}$} (A1_2);
    \path (A1_1) edge [->]node [auto] {$\scriptstyle{}$} (A1_2);
    \path (A0_0) edge [->,bend right=40]node [auto] {$\scriptstyle{v}$} (A0_2);
    \path (A0_1) edge [->]node [auto] {$\scriptstyle{}$} (A0_2);
  \end{tikzpicture}
  \]
where $\Gamma$ is a finite $k$-gerbe and the square is $2$-Cartesian. We conclude by showing that $\Delta$ is a finite gerbe over $k$.

The map $\overline F\colon \sB_{\F_p}(\GL_{n,\F_p})\arr \sB_{\F_p}(\GL_{n,\F_p})$ is induced by the Frobenius of $\GL_{n,\F_p}$. Since this last map is a surjective group homomorphism with finite kernel it follows that $\overline F$ and therefore $\Delta\arr \Gamma$ is a finite relative gerbe. This plus the assumption that $\Gamma$ is a finite gerbe implies that $\Delta$ is a finite gerbe.
\end{proof}

We finish this section by explicitly showing how to associate with an \'etale degree $n$ cover an $S_n$-torsor and conversely.
Given a scheme $U$ and an \'etale cover $E\arr U$ of degree $n$ we define $\stS_{E/U}$ as the complement in $E^n$ (the product of $n$-copies of $E$ over $U$) of the open and closed subsets given by the union of the diagonals $E^{n-1}\subseteq E^n$. The symmetric group $S_n$ acts on $\stS_{E/U}$ over $U$, and by looking at the fibers of a geometric point of $U$ we see   that the natural map $$\rho: \stS_{E/U}\times S_n\longrightarrow \stS_{E/U}\times_U\stS_{E/U}$$ sending $(a,s)$ to $(a,as)$ (with $a\in \stS_{E/U}$ and  $s\in S_n$) is a surjective map between two \'etale covers of $U$ of the same degree. Thus $\rho$ is an isomorphism and $\stS_{E/U}$ becomes an $S_n$-torsor over $U$.
\begin{prop}\label{etn to bsn}
 Let $n>0$ and denote by $\Et_n$ the stack over $\Z$ of \'etale covers of degree $n$.
 Then
 \[
 \stS\colon \Et_n\arr \sB S_n\comma E/U\longmapsto \stS_{E/U}
 \]
is an equivalence. Moreover $\pr_1\colon \stS_{E/U}\arr E$ is natural and the action of $S_{n-1}$ on the last components makes it into a $S_{n-1}$-torsor. 
\end{prop}
\begin{proof}
 
The last claim follows directly from the construction. We only have to prove that $\stS$ is an equivalence. Let $I:=\{1,\dots,n\}$. There is a global object $I_\Z=I\times \Spec \Z \in \Et_n(\Z)$. Since all \'etale covers of degree $n$ of a scheme $U$ became locally a disjoint union of copies of the base, it follows that $\Et_n$ is a trivial gerbe, and more precisely, that there is an equivalence
 \[
 \sB G \arr \Et_n \text{ where } G=\Autsh_{\Et_n}(I_\Z)
 \]
 mapping the trivial torsor to $I_\Z$. By the equivalence between neutral affine gerbes and affine group schemes, the composition $\sB G\arr \Et_n \arrdi\stS \sB S_n$ is induced by a morphism of group schemes $\phi:G=\Autsh_{\Et_n}(I_\Z)\arr S_{n,\Z}$. We must show it is an isomorphism. It is easy to see that $G$ sends any scheme $U$ to the set of continuous maps $|U|\arr S_n$, where $|U|$ is the underlying topological space of $U$ and $S_{n}$ is equipped with the discrete topology. Thus $G$ is isomorphic to the constant group scheme $S_{n,\Z}$.  The composition of the isomorphism $S_{n,\Z} = G(\Spec(\Z))\times\Spec(\Z)\cong G$ defined by the inclusion of $\Z$-rational points  with $\phi:G\arr S_{n,\Z}$ is the identity. Thus  $\phi$ is an isomorphism.
\end{proof}

\section{Theorem \ref{Dos Santos and Biswas} and \ref{Dos Santos and Biswas0}}

\begin{proof}[Proof of Theorem \ref{Dos Santos and Biswas} and Theorem \ref{Dos Santos and Biswas0}]
Without loss of generality we can assume that $\Hl^0(\odi\stX)$ is $k$, so that, in particular, $\stX$ is inflexible (see \ref{inflexible}).

If $V\in \EF(\Vect(\sX))$ then by \ref{inflexible} $V$ is the pullback along a morphism $\sX\arr \Gamma$ of a vector bundle on $\Gamma$, where $\Gamma$ is a finite gerbe. Choose a point $\Spec(k')\arr \Gamma$ where $k'/k$ is a finite field extension. Then $V$ is trivialized by the finite flat morphism $\sX\times_{\Gamma} \Spec(k')\arr \sX$. 

Consider now a map $f\colon \stY\arr\stX$ as in the statement, $V\in \Vect(\stX)$ and write $f'\colon \stY'=\Spec (f_*\odi\stY)\arr\stX$ and $r=\rk V$. Since $\Isosh_\stX(V,\odi\stX^r)\arr \stX$ is affine we have that $f^*V$ is free if and only if $f'^*V$ is free. Since by hypothesis $f_*\sO_{\sY}$ is coherent, we can therefore assume that $f\colon\stY\arr\stX$ is a finite map.

We now prove Theorem \ref{Dos Santos and Biswas0}, (1) assuming Theorem \ref{Dos Santos and Biswas}. We can assume that $\stY$ is reduced. Since by Theorem \ref{Dos Santos and Biswas} $\Vect(\sX)_f$ is contained in
the tannakian category $\EF(\Vect(\sX))$, to prove that $\Vect(\sX)_f$ is tannakian it is enough to prove that $\Vect(\sX)_f$ is stable under taking tensor products, dual, kernels and cokernels. Tensor products and dual can be easily checked, while kernels and cokernels are reduced to check the following: given $V\in\Vect(\sX)_f$ and an embedding $W\hookrightarrow V$ or a quotient $V\twoheadrightarrow W$ in $\EF(\Vect(\sX))$ we have $W\in\Vect(\sX)_f$. Write $\sY=\coprod_{i\in I} \sY_i$ for the decomposition into connected components. Each $\sY_i$ is a reduced connected algebraic stack of finite typer over $k$ and, since $\stX$ is strongly pseudo-proper, it is pseudo-proper. Thus $\Hl^0(\sY_i,\sO_{\sY_i})=k'$ is a finite field extension of $k$, and by \cite[Thoerem 4.4, pp. 13]{TZ} $\sY_i$ is inflexible and pseudo-proper over $k'$. Thus $\EF(\Vect(\sY_i))$ is a tannakian category, so the pullback of $W$ along $\sY_i\arr \sX$ is a subobject or a quotient object  of a trivial object, and consequently it is trivial itself. Thus $f^*W$ is free on each component $\stY_i$ and of the correct rank $\rk W$, which means that $f^*W$ is free on $\stY$.

We now come back to the proof of Theorem \ref{Dos Santos and Biswas} and of Theorem \ref{Dos Santos and Biswas0}, (2).
We show first how we can reduce the proof of Theorem \ref{Dos Santos and Biswas} to the case that the generic fiber of $f$ is \'etale.
Let $h\colon U\arr\stX$ be a smooth atlas and $\xi\in U$ be a generic point such that $h(\xi)$ is the generic point of $\stX$. Set $K=k(\xi)$ and consider $h(\xi)\colon \Spec K\arr \stX$ and $\Spec A =\Spec(K)\times_{\sX}\sY$. Notice that $A\neq 0$ because $\stY\arr\stX$ is surjective.
 By \cite[Lemma 2.3, pp. 8]{TZ} there exists $i\geq 0$ such that all residue fields of $A^{(i,K)}$, the $i$-th Frobenius twist of $A$, are separable over $K$. We have the following Cartesian diagrams
   \[
  \begin{tikzpicture}[xscale=2.3,yscale=-1.2]
    \node (A0_0) at (0, 0) {$\Spec B$};
    \node (A0_1) at (1, 0) {$\stY_{\textup{red}}^{(i,\stX)}$};
    \node (A1_0) at (0, 1) {$\Spec(A^{(i,K)})$};
    \node (A1_1) at (1, 1) {$\stY^{(i,\stX)}$};
    \node (A1_2) at (2, 1) {$\stY$};
    \node (A2_0) at (0, 2) {$\Spec K$};
    \node (A2_1) at (1, 2) {$\stX$};
    \node (A2_2) at (2, 2) {$\stX$};
    \path (A0_1) edge [->]node [auto] {$\scriptstyle{}$} (A1_1);
    \path (A0_0) edge [->]node [auto] {$\scriptstyle{}$} (A0_1);
    \path (A2_0) edge [->]node [auto] {$\scriptstyle{h(\xi)}$} (A2_1);
    \path (A1_0) edge [->]node [auto] {$\scriptstyle{}$} (A1_1);
    \path (A1_1) edge [->]node [auto] {$\scriptstyle{}$} (A1_2);
    \path (A1_0) edge [->]node [auto] {$\scriptstyle{}$} (A2_0);
    \path (A1_1) edge [->]node [auto] {$\scriptstyle{}$} (A2_1);
    \path (A0_0) edge [->]node [auto] {$\scriptstyle{}$} (A1_0);
    \path (A2_1) edge [->]node [auto] {$\scriptstyle{F^i}$} (A2_2);
    \path (A1_2) edge [->]node [auto] {$\scriptstyle{f}$} (A2_2);
  \end{tikzpicture}
  \]
where $(-)_\red$ denotes the reduction and $F$ is the absolute Frobenius of $\stX$. Set $f'\colon \sY_{\red}^{(i,\sX)}\arr \stX$ for the composition. Since $h(\xi)$ is the composition of a smooth map $U\arr \stX$ and a generic point $\Spec k(\xi)\arr U$, we can conclude that $B$ is reduced. Since it is a quotient of $A^{(i,K)}$ we can conclude that $B/K$ is \'etale over $K$, that is, the generic fiber of $f'$ is \'etale.
Notice that if $f^*V$ is free then $f'^*(F^{i*}V)$ is free and, by \ref{Frobenius pullback of essentially finite}, $V$ is essentially finite if  ${F^i}^*V$ is essentially finite.
Thus we can replace $\stY\arr\stX$ by $\stY_{\textup{red}}^{(i,\stX)}\arr\stX$ and assume that the generic fiber of $f$ is \'etale of degree say $n\geq 0$.

For each smooth map $U\arr\sX$, where $U$ is a connected scheme, let $Y_U:=\sY\times_{\sX}U$. Since the generic point of $U$ goes to the generic point of $\sX$, the generic fibre $L_U$ of $Y_U\arr U$ is finite \'etale of degree $n$. By \ref{etn to bsn} $L_U$ corresponds to  an $S_n$-torsor $\stS_{L_U/K(U)}$ over $\Spec(K(U))$, where $K(U)$ is the function field of $U$. Let $\Theta_{U}$ be the  normalization of $U$ inside $S_{L_U/K(U)}=\Hl^0(\stS_{L_U/K(U)})$. Then $\Theta_{U}$ is equipped with an action of $S_n$ and  a morphism $\Theta_U\arr Y_U$ define by the projection $\sS_{L_U/K(U)}\arr \Spec(L_U)$ (the last claim of \ref{etn to bsn}). This construction is functorial.
If $V\arr U$ is a smooth morphism then by \ref{etn to bsn} and \cite[\href{http://stacks.math.columbia.edu/tag/03GC}{03GC}]{SP} we have $$s_{U,V}:S_{L_U/K(U)}\otimes_{K(U)} K(V)\cong S_{L_V/K(V)},\hspace{15pt} \theta_{U,V}: (\Theta_U\times_UV\arr Y_U\times_UV)\cong(\Theta_V\arr Y_V) $$ and also the action of $S_n$ on $\Theta_V$ is the same as the one obtained by the pullback of the action on $\Theta_U$. Moreover, if 
there is a third map $W\arr V$, then $s_{U,V}, s_{V,W}, s_{U,W}$ and $\theta_{U,V}, \theta_{V,W}, \theta_{U,W}$ are compatible in a natural way. This allow us to construct maps $\Theta_\stX\arr\stZ\arr \stX$ fitting in $2$-Cartesian diagrams 
  \[
  \begin{tikzpicture}[xscale=2.3,yscale=-1.2]
    \node (A0_0) at (0, 0) {$\Theta_U$};
    \node (A0_1) at (1, 0) {$\Theta_\stX$};
    \node (A0_2) at (2, 0) {$\Spec k$};
    \node (A1_0) at (0, 1) {$[\Theta_U/S_n]$};
    \node (A1_1) at (1, 1) {$\stZ$};
    \node (A1_2) at (2, 1) {$\sB_k S_n$};
    \node (A2_0) at (0, 2) {$U$};
    \node (A2_1) at (1, 2) {$\stX$};
    \path (A0_1) edge [->]node [auto] {$\scriptstyle{\pi}$} (A1_1);
    \path (A0_0) edge [->]node [auto] {$\scriptstyle{}$} (A0_1);
    \path (A2_0) edge [->]node [auto] {$\scriptstyle{}$} (A2_1);
    \path (A1_0) edge [->]node [auto] {$\scriptstyle{}$} (A1_1);
    \path (A0_2) edge [->]node [auto] {$\scriptstyle{}$} (A1_2);
    \path (A1_1) edge [->]node [auto] {$\scriptstyle{}$} (A1_2);
    \path (A1_0) edge [->]node [auto] {$\scriptstyle{}$} (A2_0);
    \path (A1_1) edge [->]node [auto] {$\scriptstyle{h}$} (A2_1);
    \path (A0_0) edge [->]node [auto] {$\scriptstyle{}$} (A1_0);
    \path (A0_1) edge [->]node [auto] {$\scriptstyle{}$} (A0_2);
  \end{tikzpicture}
  \]
  and a map $\Theta_{\sX}\to \sY$ fitting in the 2-diagram
 $$\xymatrix{\Theta_{\sX}\ar[rr]\ar[rd]_-{h\pi}&& \sY\ar[ld]^-f\\ &\sX &}$$ for all smooth $U\arr \stX$ with $U$ being connected, where $\sZ$ is the following category fibred over $\sX$: for each $t:T\arr\sX$, $\sZ(T)$ is the category consisting of diagrams 
 \[
\xymatrix{
P \ar[dr]_y \ar[r]^x &
\Theta_T \ar[d] \ar[r] &\Theta_{\sX}\ar[d]^{h\pi} \\
&T \ar[r]^t &\sX}
\]
where $y: P\arr T$ is a torsor under $S_n$, $\Theta_T$ is the pullback, and $x$ is a map of schemes which is $S_n$-equivariant. Just as in the case when $\Theta_{\sX}$ is a scheme and $\sZ=[\Theta_{\sX}/S_n]$, one can show that the diagram $$\xymatrix{
\Theta_\sX \ar[d]^\pi \ar[r] &\Spec k\ar[d] \\
\sZ \ar[r] &\sB_kS_n}$$
is Cartesian. In particular $\stZ$ is an algebraic stack of finite type over $k$.
Since $\Vect(\sX)_f\subseteq \Vect(\sX)_{h\pi}$ is fully faithful, by \cite[Remark B7]{TZ} we can assume $\stY=\Theta_\stX$ and $f=h\pi$ for both Theorem \ref{Dos Santos and Biswas} and Theorem \ref{Dos Santos and Biswas0}, (2).

Consider the following $2$-Cartesian diagram
  \[
  \begin{tikzpicture}[xscale=2.3,yscale=-1.2]
    \node (A0_0) at (0, 0) {$\Theta_\stX$};
    \node (A0_1) at (1, 0) {$\stX$};
    \node (A0_2) at (2, 0) {$\Spec k$};
    \node (A1_0) at (0, 1) {$\stZ$};
    \node (A1_1) at (1, 1) {$\stX\times\sB_k S_n$};
    \node (A1_2) at (2, 1) {$\sB_k S_n$};
    \node (A2_1) at (1, 2) {$\stX$};
    \node (A2_2) at (2, 2) {$\Spec k$};
    \path (A1_0) edge [->]node [auto,swap] {$\scriptstyle{h}$} (A2_1);
    \path (A0_0) edge [->]node [auto] {$\scriptstyle{f}$} (A0_1);
    \path (A2_1) edge [->]node [auto] {$\scriptstyle{}$} (A2_2);
    \path (A1_0) edge [->]node [auto] {$\scriptstyle{u}$} (A1_1);
    \path (A0_1) edge [->]node [auto] {$\scriptstyle{}$} (A0_2);
    \path (A1_1) edge [->]node [auto] {$\scriptstyle{}$} (A1_2);
    \path (A0_2) edge [->]node [auto] {$\scriptstyle{}$} (A1_2);
    \path (A1_1) edge [->]node [auto] {$\scriptstyle{v}$} (A2_1);
    \path (A0_0) edge [->]node [auto] {$\scriptstyle{\pi}$} (A1_0);
    \path (A0_1) edge [->]node [auto] {$\scriptstyle{}$} (A1_1);
    \path (A1_2) edge [->]node [auto] {$\scriptstyle{}$} (A2_2);
  \end{tikzpicture}
  \]
Since $u$ is finite and $v$ is proper we can conclude that the map $h\colon \stZ\arr\stX$ maps coherent sheaves to coherent sheaves. Thus $\stZ$ is a strongly pseudo-proper and normal algebraic stack of finite type over $k$. We are going to show that $\odi\stX\arr h_*\odi\stZ$ is an isomorphism. In particular it will follow that $\Hl^0(\odi\stZ)=k$ and therefore that $\stZ$ is inflexible over $k$.
%
%
Given $U=\Spec A\arr \stX$ a smooth map with $U$ being connected and written $\Theta_U=\Spec B$, the push forward of the structure sheaf along $[\Theta_U/S_n]\arr U$ is $B^{S^n}$. and we must show that $B^{S^n}\subseteq A$. Set $K_A$ for the function field of $A$ and $K_B$ for the generic fibre of $B/A$, that is $K_B=S_{L_U/K_A}$. We have $B^{S_n}\subseteq B\bigcap K_B^{S_n}=B\bigcap K_A=A$ because $A$ is normal and $K_B$ is an $S_n$-torsor over $K_A$.

Let $V$ be a vector bundle on $\sX$ which is trivialized by $f$. Since $\stZ$ is pseudo-proper and inflexible, by \ref{finite etale trivialization} $h^*V$ is essentially finite. Thus there is a finite gerbe $\Gamma$ and a 2-commutative diagram $$\xymatrix{\stZ\ar[d]^-h\ar[rr]&&\Gamma\ar[d]^{\lambda}\\ \sX\ar[rr]^-{\phi}&& \sB{\rm GL}_m}$$  where $\phi$  corresponds to the vector bundle $V$. Replacing $\Gamma$ by its  image under $\lambda$ we may assume that $\lambda$ is faithful. Since $\Gamma$ is finite it follows that $\lambda$ is affine thanks to \cite[Remark B7]{TZ}. As $h_*\sO_{\stZ}=\sO_{\sX}$ the unique map $\stZ\arr \sX\times_{\sB{\rm GL}_m}\Gamma$ induces 
$$
\Spec (h_*\odi\stZ )=\stX\arr  \sX\times_{\sB{\rm GL}_m}\Gamma\arr \Gamma
$$
Then $\phi$ factors though $\lambda: \Gamma\arr \sB{\rm GL}_m$, so that $V$ is essentially finite. This ends the proof of Theorem \ref{Dos Santos and Biswas}.

For thereom \ref{Dos Santos and Biswas0}, since $h_*\odi\stZ= \odi\stX$, we have that the pullback $\Vect(\stX)\arr\Vect(\stZ)$ is fully faithful (see \cite[Lemma 7.17]{BV}). In particular we get a fully faithful monoidal map $\Vect(\stX)_f \arr \Vect(\stZ)_\pi$. By \ref{finite etale trivialization} $\Vect(\stZ)_\pi$ corresponds to a gerbe $\Gamma_\pi$ affine over $\sB_k S_n$ and thus finite and \'etale. By \cite[Remark B7]{TZ} it follows that the gerbe associated with $\Vect(\stX)_f$ is a quotient of $\Gamma_\pi$ as required. 
\end{proof}

\end{document}